\documentclass[12pt]{article}

\usepackage{amssymb}
\usepackage{mathtools}
\usepackage{tikz}
\usetikzlibrary{matrix,arrows,decorations.pathmorphing}
\usepackage{amsmath,color}

\usepackage[T1]{fontenc} 
\usepackage{lmodern}\usepackage{graphicx}
\usepackage{microtype} 

\usepackage[hmarginratio=1:1,top=32mm,columnsep=20pt]{geometry} 
\usepackage{hyperref}
\usepackage{amsthm}
\usepackage{url}

\newtheorem{thm}{Theorem}[section]
\newtheorem{cor}[thm]{Corollary}
\newtheorem{lem}[thm]{Lemma}
\newtheorem{ex}[thm]{Example}
\newtheorem{prop}[thm]{Proposition}
\newtheorem{defn}[thm]{Definition}
\newtheorem{rem}[thm]{Observation}

\makeindex
\title{\vspace{-15mm}\fontsize{24pt}{10pt}\selectfont\textbf{Universal enveloping algebras for Malcev color algebras}} 

\author{
\large
\textsc{Daniel de la Concepci\'on}\thanks{FPU12 Grant from the Ministerio de Educaci\'on, Cultura y Deporte in Spain; and associated to the research project MTM2013-45588-C3-3-P}\\[2mm] 
\normalsize Universidad de la Rioja \\
Departamento de Matem\'aticas y Computaci\'on\\ 
\vspace{-5mm}
}
\date{}

\begin{document}

\maketitle 
\abstract{In this paper we give a construction of the universal enveloping algebra of a Malcev algebra
in categories of group algebra comodules with a symmetry given by a bicharacter of the group. A particular example of such categories is the category of super vector spaces.}

\section{Motivation}

There is a very well known relation between connected Hopf algebras and Lie algebras; as it is shown in \cite{Mil65}. In fact this categories are equivalent in the case of having a field of characteristic zero. This relation has been extended to connected cocommutative and coassociative bialgebras and Sabinin algebras in \cite{UnivSab}.

It has been shown that this relation between connected Hopf algebras and Lie algebras is deeper, as it stands when moving to categories other than $\mathbf{Vec}$. Specifically, the monoidal categories that have a forgetful functor to $\mathbf{Vec}$ and a braiding. A survey can be read in \cite{Shes12}, where the most common examples are explained. This relations makes us wonder if this results can be extended to wider subclasses of Sabinin algebras in this types of categories.

 One easy and general subclass of this categories mentioned in the previous paragraph, are the categories of group comodules with a symmetry given by a group bicharacter. Results about this structures in relation with Lie algebras can be read in \cite{SchColLie}, \cite{Ritt79} and \cite{Bah92}, originally.

In this paper we generalize this relation with the easiest examples: Malcev algebras on $\mathbb{K}[G]$-comodule categories; using already known category theoretic tools.

\section{Category theory}

In this first section, we recall some category theory and define the categorical objects that will appear through the paper.

Also, everywhere in this papaer, unless otherwise stated, $\mathbb{K}$ is a field of characteristic zero and any group $G$ is abelian.

\subsection{Categorical structures}

Firstly, we define some categorical structures through commutative diagrams; that in the next sections will be used.

\begin{defn}
A category $\mathcal{C}$ is called monoidal if it is equiped with:
\begin{itemize}
\item A bi-functor $\otimes:\mathcal{C}\times\mathcal{C}\to\mathcal{C}$ called tensor product.
\item An object $\mathbb{I}$ called the unit element.
\item A natural isomorphism $\alpha_{A,B,C}:(A\otimes B)\otimes C\to A\otimes(B\otimes C)$ called associator.
\item Two natural isomorphisms $\lambda_A:\mathbb{I}\otimes A\to A$ and $\rho_A:A\otimes\mathbb{I}\to A$, respectively called left and right unitor.
\end{itemize}
such that the diagrams \ref{fig:pentagon} and \ref{fig:triangle} are commutative.
\end{defn}

The first examples of monoidal categories are those with finite products or coproducts:

\begin{ex}\label{prodtensor}
Let $\mathcal{C}$ be a category with finite products. Then $(\mathcal{C},\times,T,\alpha,\rho,\lambda)$ with $T$ as the terminal object of the category and $\times$ as the categorical product is a monoidal category given the following:
\begin{itemize}
\item $\alpha_{A,B,C}$ as the morphisms obtained by recombining the projections from $(A\times B)\times C$ to match $A\times (B\times C)$.
\item $T$ is the unit element.
\item $\lambda_A:A\times T\to A$ and $\rho_A:T\times A\to A$ as the projections.
\end{itemize}
\end{ex}

\begin{ex}
Let $\mathcal{C}$ be a category with finite coproducts. Then $(\mathcal{C},\oplus,\bot,\alpha,\rho,\lambda)$ with $\bot$ as the initial object of the category and $\oplus$ as the categorical coproduct is a monoidal category given the following:
\begin{itemize}
\item $\alpha_{A,B,C}$ as the morphisms obtained by recombining the injections to $(A\oplus B)\oplus C$ to match $A\oplus (B\oplus C)$.
\item $\bot$ as the unit element.
\item $\lambda_A$ and $\rho_A$ as the morphisms obtained by the coproduct from the identity on $A$ and the only morphism $f_A:\bot\to A$.
\end{itemize}
\end{ex}

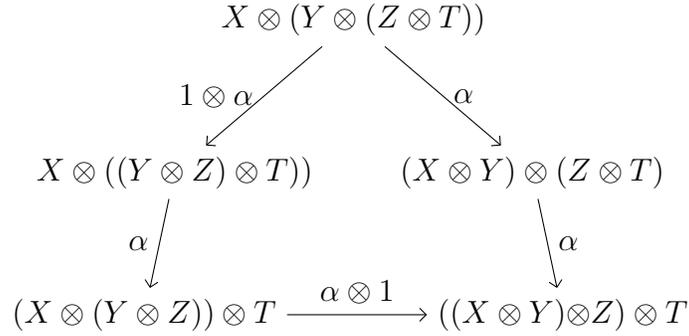
\begin{figure}
\centering
\begin{tikzpicture}
\node (P0) at (90:2.8cm) {$X\otimes (Y\otimes (Z\otimes T))$};
\node (P1) at (90+72:2.5cm) {$X\otimes ((Y\otimes Z)\otimes T))$} ;
\node (P2) at (90+72+40:3cm) {${(X\otimes (Y\otimes Z))}\otimes T$};
\node (P3) at (90+4*72-40:3cm) {$((X\otimes Y){\otimes Z)\otimes T}$};
\node (P4) at (90+4*72:2.5cm) {$(X\otimes Y)\otimes (Z\otimes T)$};
\draw
(P0) edge[->,>=angle 90] node[left] {$1\otimes\alpha$} (P1)
(P1) edge[->,>=angle 90] node[left] {$\alpha$} (P2)
(P2) edge[->,>=angle 90] node[above] {$\alpha\otimes 1$} (P3)
(P4) edge[->,>=angle 90] node[right] {$\alpha$} (P3)
(P0) edge[->,>=angle 90] node[right] {$\alpha$} (P4);
\end{tikzpicture}
\caption{Pentagon diagram}
\label{fig:pentagon}
\end{figure}

\begin{figure}
\centering
\begin{tikzpicture}
\node (P0) at (30:2.5cm) {$X\otimes (\mathbb{I}\otimes Y)$};
\node (P1) at (150:2.5cm) {$(X\otimes \mathbb{I})\otimes Y$} ;
\node (P2) at (270:1cm) {$X\otimes Y$};
\draw
(P1) edge[->,>=angle 90] node[above] {$\alpha$} (P0)
(P1) edge[->,>=angle 90] node[left] {$\rho\otimes id$} (P2)
(P0) edge[->,>=angle 90] node[right] {$id\otimes \lambda$} (P2);
\end{tikzpicture}
			\caption{Triangle diagram}
	\label{fig:triangle}
\end{figure}
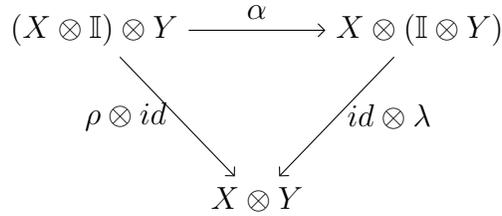

The next step is to consider isomorphisms between $A\otimes B$ and $B\otimes A$, this type of categories are those denoted as braided categories.

\begin{defn}
A natural isomorphism $c_{A,B}:A\otimes B\to B\otimes A$  for $(\mathcal{C},\otimes,\mathbb{I},\alpha,\rho,\lambda)$ is called a braiding in case the hexagon diagrams \ref{fig:hexagon1} and \ref{fig:hexagon2} commute. In this case $(\mathcal{C},\otimes,\mathbb{I},c,\alpha,\rho,\lambda)$ is called a braided monoidal category.

In case $c_{A,B}\circ c_{B,A}=id_{B\otimes A}$, the braiding is called symmetric and the resulting category is a symmetric monoidal category.

\begin{figure}
	\centering
\begin{tikzpicture}
\node (P0) at (40:2.5cm) {$(B\otimes C)\otimes A$};
\node (P1) at (140:2.5cm) {$A\otimes (B\otimes C)$} ;
\node (P2) at (180:2.5cm) {$(A\otimes B)\otimes C$};
\node (P3) at (220:2.5cm) {$(B\otimes A)\otimes C$};
\node (P4) at (320:2.5cm) {$B\otimes (A\otimes C)$};
\node (P5) at (360:2.5cm) {$B\otimes (C\otimes A)$};
\draw
(P1) edge[->,>=angle 90] node[above] {$c$} (P0)
(P0) edge[->,>=angle 90] node[right] {$\alpha$} (P5)
(P2) edge[->,>=angle 90] node[left] {$\alpha$} (P1)
(P3) edge[->,>=angle 90] node[below] {$\alpha$} (P4)
(P4) edge[->,>=angle 90] node[right] {$id\otimes c$} (P5)
(P2) edge[->,>=angle 90] node[left] {$c\otimes id$} (P3);
\end{tikzpicture}
	\caption{First hexagon diagram}
	\label{fig:hexagon1}
\end{figure}
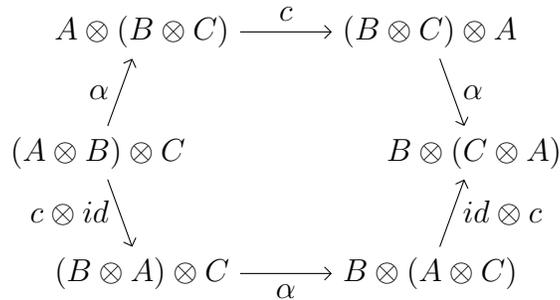

\begin{figure}
	\centering
\begin{tikzpicture}
\node (P0) at (40:2.5cm) {$C\otimes (A\otimes B)$};
\node (P1) at (140:2.5cm) {$(A\otimes B)\otimes C$} ;
\node (P2) at (180:2.5cm) {$A\otimes (B\otimes C)$};
\node (P3) at (220:2.5cm) {$A\otimes(C\otimes B)$};
\node (P5) at (360:2.5cm) {$(C\otimes A)\otimes B$};
\node (P4) at (320:2.5cm) {$(A\otimes C)\otimes B$};
\draw
(P1) edge[->,>=angle 90] node[above] {$c$} (P0)
(P0) edge[->,>=angle 90] node[right] {$\alpha$} (P5)
(P2) edge[->,>=angle 90] node[left] {$\alpha$} (P1)
(P3) edge[->,>=angle 90] node[below] {$\alpha$} (P4)
(P4) edge[->,>=angle 90] node[right] {$c\otimes id$} (P5)
(P2) edge[->,>=angle 90] node[left] {$id\otimes c$} (P3);
\end{tikzpicture}
	\caption{Second hexagon diagram}
	\label{fig:hexagon2}
\end{figure}

\end{defn}

It is obvious to wonder if there are a special type of functors between monoidal categories; i.e, functors that `preserve' the monoidal structure in some sense. 

\begin{defn}
A functor $F:(\mathcal{C},\bullet,\mathbb{I}_C)\to(\mathcal{D},\otimes,\mathbb{I}_D)$ between two monoidal categories is called monoidal if it is equiped with:
\begin{itemize}
\item $\phi:(FA)\otimes (FB)\to F(A\bullet B)$ a natural transformation.
\item $r:\mathbb{I}_D\to F\mathbb{I}_C$ another natural tranformation.
\end{itemize}
such that  the diagrams \ref{fig:monoidal functor1}, \ref{fig:monoidal functor2} and \ref{fig:monoidal functor3} commute.

In case there is a braiding, the functor $F$ is called braided if also the diagram \ref{fig:braided monoidal functor} commutes.
\end{defn}

\begin{figure}
	\centering
\begin{tikzpicture}
\node (P0) at (40:2.5cm) {$FA\otimes (FB\otimes FC)$};
\node (P1) at (140:2.5cm) {$(FA\otimes FB)\otimes FC$} ;
\node (P2) at (180:2.5cm) {$F(A\bullet B)\otimes FC$};
\node (P3) at (220:2.5cm) {$F((A\bullet B)\bullet C)$};
\node (P4) at (320:2.5cm) {$F(A\bullet (B\bullet C))$};
\node (P5) at (360:2.5cm) {$FA\otimes F(B\bullet C)$};
\draw
(P1) edge[->,>=angle 90] node[above] {$\alpha$} (P0)
(P0) edge[->,>=angle 90] node[right] {$id\otimes \phi$} (P5)
(P1) edge[->,>=angle 90] node[left] {$\phi\otimes id$} (P2)
(P3) edge[->,>=angle 90] node[below] {$F\alpha$} (P4)
(P5) edge[->,>=angle 90] node[right] {$\phi$} (P4)
(P2) edge[->,>=angle 90] node[left] {$\phi$} (P3);
\end{tikzpicture}
	\caption{First monoidal functor diagram}
	\label{fig:monoidal functor1}
\end{figure}

\begin{figure}
	\centering
\begin{tikzpicture}
\node (P0) at (45:2cm) {$F\mathbb{I}_\mathcal{C}\otimes FB$};
\node (P1) at (135:2cm) {$\mathbb{I}_\mathcal{D}\otimes FB$} ;
\node (P2) at (225:2cm) {$FB$};
\node (P3) at (315:2cm) {$F(\mathbb{I}_\mathcal{C}\bullet B)$};
\draw
(P1) edge[->,>=angle 90] node[above] {$r\otimes id$} (P0)
(P0) edge[->,>=angle 90] node[right] {$\phi$} (P3)
(P3) edge[->,>=angle 90] node[below] {$F\lambda$} (P2)
(P1) edge[->,>=angle 90] node[right] {$\lambda$} (P2);
\end{tikzpicture}
	\caption{Second monoidal functor diagram}
	\label{fig:monoidal functor2}
\end{figure}

\begin{figure}
	\centering
\begin{tikzpicture}
\node (P0) at (45:2cm) {$FA\otimes F\mathbb{I}_\mathcal{C}$};
\node (P1) at (135:2cm) {$FA\otimes \mathbb{I}_\mathcal{D}$} ;
\node (P2) at (225:2cm) {$FA$};
\node (P3) at (315:2cm) {$F(A\bullet \mathbb{I}_\mathcal{C})$};
\draw
(P1) edge[->,>=angle 90] node[above] {$id\otimes r$} (P0)
(P0) edge[->,>=angle 90] node[right] {$\phi$} (P3)
(P3) edge[->,>=angle 90] node[above] {$F\rho$} (P2)
(P1) edge[->,>=angle 90] node[right] {$\rho$} (P2);
\end{tikzpicture}
	\caption{Third monoidal functor diagram}
	\label{fig:monoidal functor3}
\end{figure}

\begin{figure}
	\centering
\begin{tikzpicture}
\node (P0) at (45:2cm) {$FB\otimes FA$};
\node (P1) at (135:2cm) {$FA\otimes FB$} ;
\node (P2) at (225:2cm) {$F(A\bullet B)$};
\node (P3) at (315:2cm) {$F(B\bullet A)$};
\draw
(P1) edge[->,>=angle 90] node[above] {$c$} (P0)
(P0) edge[->,>=angle 90] node[right] {$\phi$} (P3)
(P2) edge[->,>=angle 90] node[above] {$Fc$} (P3)
(P1) edge[->,>=angle 90] node[right] {$\phi$} (P2);
\end{tikzpicture}
	\caption{Braided monoidal functor diagram}
	\label{fig:braided monoidal functor}
\end{figure}

 \begin{defn}
The dual concept of a monoidal functor is called a comonoidal functor. In other words, a comonoidal functor is a functor $F:(\mathcal{C},\bullet,\mathbb{I}_C)\to (\mathcal{D},\otimes,\mathbb{I}_D)$ between two monoidal cateogries equiped with:
\begin{itemize}
\item $\phi:F(A\bullet B)\to F(A)\otimes F(B)$ a natural transformation.
\item $r:F\mathbb{I}_C\to \mathbb{I}_D$ a natural transformation.
\end{itemize}
such that the dual diagrams of \ref{fig:monoidal functor1}, \ref{fig:monoidal functor2} and \ref{fig:monoidal functor3} commute. 

In case there is a braiding, the functor $F$ is called braided if also the diagram dual to  \ref{fig:braided monoidal functor} commutes.
\end{defn}

Sometimes, the more natural framework is one where the morphisms $Hom(A,B)$ is an object in some other category and the composition morphism is an arrow in that other category. For example, in $\mathbf{Vec}$, $Hom(V,W)$ is a vector space and the composition of two such morphisms is bilinear. The general definition of this situation is the following:

\begin{defn}Given $(\mathcal{M},\otimes,\mathbb{I},\alpha,\rho,\lambda)$ a monoidal category. $\mathcal{C}$ is called an enriched category over $\mathcal{M}$ if there is:
\begin{itemize}
\item A family of objects $Obj(\mathcal{C})$
\item $\forall A,B\in Obj(\mathcal{C})$, an object in $\mathcal{M}$, $Hom_\mathcal{C}(A,B)=\mathcal{C}(A,B)\in\mathcal{M}$ 
\item $\forall A,B,C\in Obj(\mathcal{C})$, a composition map $\circ:\mathcal{C}(B,C)\otimes \mathcal{C}(A,B)\to\mathcal{C}(A,C)$, an arrow  in $\mathcal{M}$
\item $\forall A\in Obj(\mathcal{C})$, a unity map $u_A:\mathbb{I}\to \mathcal{C}(A,A)$, an arrow in $\mathcal{M}$
\end{itemize}
And all of them make the diagrams \ref{fig:enriched1}, \ref{fig:enriched2} and \ref{fig:enriched3} commute. In case $\mathcal{M}$ is the category of abelian groups, we call $\mathcal{C}$ a preadditive category.
In case $\mathcal{M}$ is the category of vector spaces over $\mathbb{K}$ with natural tensor product, we call $\mathcal{C}$ a $\mathbb{K}$-linear category. In general, $\mathcal{C}$ is called an $\mathcal{M}$-category.
\end{defn}

\begin{figure}
	\centering
\begin{tikzpicture}
\node (P0) at (90:2cm) {$\mathcal{C}(A,B)$};
\node (P1) at (90+72:2.5cm) {$\mathcal{C}(B,D)\otimes \mathcal{C}(A,B)$} ;
\node (P2) at (90+72+40:4cm) {${(\mathcal{C}(C,D)\otimes\mathcal{C}(B,C))}\otimes \mathcal{C}(A,B)$};
\node (P3) at (90+4*72-40:4cm) {$\mathcal{C}(C,D)\otimes{(\mathcal{C}(B,C)\otimes \mathcal{C}(A,B))}$};
\node (P4) at (90+4*72:2.5cm) {$\mathcal{C}(C,D)\otimes\mathcal{C}(A,C)$};
\draw
(P1) edge[->,>=angle 90] node[left] {$\circ$} (P0)
(P2) edge[->,>=angle 90] node[left] {$\circ\otimes id$} (P1)
(P2) edge[->,>=angle 90] node[above] {$\alpha$} (P3)
(P3) edge[->,>=angle 90] node[right] {$id\otimes \circ$} (P4)
(P4) edge[->,>=angle 90] node[right] {$\circ$} (P0);
\end{tikzpicture}
	\caption{First enriched category diagram}
	\label{fig:enriched1}
\end{figure}

\begin{figure}
	\centering
\begin{tikzpicture}
\node (P0) at (30:2.5cm) {$\mathcal{C}(B,B)\otimes \mathcal{C}(A,B)$};
\node (P1) at (150:2.5cm) {$\mathbb{I}\otimes \mathcal{C}(A,B)$} ;
\node (P2) at (270:1cm) {$\mathcal{C}(A,B)$};
\draw
(P1) edge[->,>=angle 90] node[above] {$u \otimes id$} (P0)
(P0) edge[->,>=angle 90] node[right] {$\circ$} (P2)
(P1) edge[->,>=angle 90] node[left] {$\lambda$} (P2);
\end{tikzpicture}
	\caption{Second enriched category diagram}
	\label{fig:enriched2}
\end{figure}

\begin{figure}
	\centering
\begin{tikzpicture}
\node (P0) at (30:2.5cm) {$\mathcal{C}(A,B)\otimes \mathcal{C}(B,B)$};
\node (P1) at (150:2.5cm) {$\mathcal{C}(A,B)\otimes \mathbb{I}$} ;
\node (P2) at (270:1cm) {$\mathcal{C}(A,B)$};
\draw
(P1) edge[->,>=angle 90] node[above] {$id\otimes u$} (P0)
(P1) edge[->,>=angle 90] node[left] {$\rho$} (P2)
(P0) edge[->,>=angle 90] node[right] {$\circ$} (P2);
\end{tikzpicture}
	\caption{Third enriched category diagram}
	\label{fig:enriched3}
\end{figure}

There exist several examples of enriched categories:
\begin{ex}
A natural first example is the $\mathbf{Set}$-categories. This are just categories since rewriting the definition of enriched category with $\mathcal{M}=\mathbf{Set}$ the definition of a category is reached.
\end{ex}
\begin{ex}
Another usefull example is a $\mathbb{K}$-linear category: the category of vector spaces over $\mathbb{K}$; i.e., $\mathbf{Vec}$.
\begin{itemize}
\item The set of linear maps between two vector spaces is a vector space.
\item The composition of linear maps defines a bilinear map; hence it defines a linear map $\circ:Hom(B,C)\otimes Hom(A,B)\to Hom(A,C)$.
\item There is a linear map $u_A:\mathbb{K}\to End(A)$ such that $u_A(1)=id(A)$.
\end{itemize}
The previous arrows make the diagrams \ref{fig:enriched1}, \ref{fig:enriched2} and \ref{fig:enriched3} commutative.
\end{ex}

Enriched categories also have enriched functors and enriched natural transformations; which are those that respect the category structure of $\mathcal{M}$. For instance, a $\mathbb{K}$-linear functor is one that acts on morphisms as $F(\lambda f+\mu g)=\lambda F(f)+\mu F(g)$; meaning that $F_{X,Y}:Hom(X,Y)\to Hom(F X, F Y)$ is a linear map.

\begin{defn}
An enriched functor between two $\mathcal{M}$-categories, $\mathcal{C}_1$ and $\mathcal{C}_2$ is defined as:
\begin{itemize}
\item A function $F:Obj(\mathcal{C}_1)\to Obj(\mathcal{C}_2)$ on the objects.
\item A collection of $\mathcal{M}$-morphisms $F_{x,y}:Hom_{\mathcal{C}_1}(x,y)\to Hom_{\mathcal{C}_2}(Fx,Fy)$
\end{itemize}
such that the diagrams \ref{fig:enriched functor1} and  \ref{fig:enriched functor2} are commutative.
\begin{figure}
	\centering
\begin{tikzpicture}
\node (P0) at (30:3cm) {$Hom_{\mathcal{C}_1}(a,c)$};
\node (P1) at (150:3cm) {$Hom_{\mathcal{C}_1}(b,c)\otimes Hom_{\mathcal{C}_1}(a,b)$} ;
\node (P2) at (210:3cm) {$Hom_{\mathcal{C}_2}(Fb,Fc)\otimes Hom_{\mathcal{C}_2}(Fa,Fb)$};
\node (P3) at (360-30:3cm) {$Hom_{\mathcal{C}_2}(Fa,Fb)$};
\draw
(P0) edge[->,>=angle 90] node[right] {$F$} (P3)
(P1) edge[->,>=angle 90] node[above] {$\circ_1$} (P0)
(P1) edge[->,>=angle 90] node[right] {$F\otimes F$} (P2)
(P3) edge[->,>=angle 90] node[below] {$\circ_2$} (P2);
\end{tikzpicture}
	\caption{First enriched functor diagram}
	\label{fig:enriched functor1}
\end{figure}

\begin{figure}
	\centering
\begin{tikzpicture}
\node (P0) at (45:2.5cm) {$Hom_{\mathcal{C}_2}(FA,FA)$};
\node (P1) at (180-45:2.5cm) {$Hom_{\mathcal{C}_1}(A,A)$} ;
\node (P2) at (270:1cm) {$\mathbb{I}$};
\draw
(P2) edge[->,>=angle 90] node[right] {$u_2$} (P0)
(P2) edge[->,>=angle 90] node[left] {$u_1$} (P1)
(P1) edge[->,>=angle 90] node[above] {$F_{A,A}$} (P0);
\end{tikzpicture}
	\caption{Second enriched functor diagram}
	\label{fig:enriched functor2}
\end{figure}
\end{defn}

\begin{defn}
An enriched natural transformation between two $\mathcal{M}$-functors, $F,G:\mathcal{C}_1\to\mathcal{C}_2$ is defined as a colletion of $\mathcal{M}$-morphisms $\mu_c:\mathbb{I}\to Hom_{\mathcal{C}_2}(Fc,Gc)$;
such that the diagram \ref{fig:enriched nat} is commutative.
\begin{figure}
	\centering
\begin{tikzpicture}
\node (P0) at (360/8:4cm) {$Hom_{\mathcal{C}_2}(Fb,Gb)\otimes Hom_{\mathcal{C}_2}(Fa,Fb)$};
\node (P1) at (180-360/8:4cm) {$\mathbb{I}\otimes Hom_{\mathcal{C}_1}(a,b)$} ;
\node (P2) at (360/2:2.83cm) {$Hom_{\mathcal{C}_1}(a,b)$};
\node (P3) at (180+360/8:4cm) {$Hom_{\mathcal{C}_1}(a,b)\otimes\mathbb{I}$};
\node (P4) at (-360/8:4cm) {$Hom_{\mathcal{C}_2}(Ga,Gb)\otimes Hom_{\mathcal{C}_2}(Fa,Ga)$};
\node (P5) at (0:2.83cm) {$Hom_{\mathcal{C}_2}(Fa,Gb)$};
\draw
(P2) edge[->,>=angle 90] node[right] {$\cong$} (P1)
(P2) edge[->,>=angle 90] node[right] {$\cong$} (P3)
(P3) edge[->,>=angle 90] node[below] {$G\otimes\mu_a$} (P4)
(P4) edge[->,>=angle 90] node[right] {$\circ_2$} (P5)
(P0) edge[->,>=angle 90] node[right] {$\circ_2$} (P5)
(P1) edge[->,>=angle 90] node[below] {$\mu_b\otimes F$} (P0);
\end{tikzpicture}
	\caption{Enriched natural transformation diagram}
	\label{fig:enriched nat}
\end{figure}
\end{defn}

The product of any two categories can be defined as $Obj(\mathcal{C}_1\times\mathcal{C}_2)=Obj(\mathcal{C}_1)\times Obj(\mathcal{C}_2)$ and $Hom((A_1,A_2),(B_1,B_2))=Hom(A_1,A_2)\times Hom(B_1,B_2)$. This definition is used to define, for example, tensor products as functors from a product category $\mathcal{C}\times\mathcal{C}$. The problem is that it doesn't make sence for $\mathcal{M}$-categories; since the morphisms no longer are defined as sets.

What would make sense is to define $Hom((A_1,A_2),(B_1,B_2))=Hom(A_1,A_2)\otimes Hom(B_1,B_2)$; since the objects can be tensored:

\begin{defn}
Let $\mathcal{C}_1$ and $\mathcal{C}_2$ be $\mathcal{M}$-categories; with $\mathcal{M}$ symmetric monoidal. We then define $\mathcal{C}_1\otimes_\mathcal{M}\mathcal{C}_2$ as the  $\mathcal{M}$-category with:
\begin{itemize}
\item Objects as $Obj(\mathcal{C}_1\otimes_\mathcal{M}\mathcal{C}_2)=Obj(\mathcal{C}_1)\times Obj(\mathcal{C}_2)$
\item Morphisms as $Hom((A_1,A_2),(B_1,B_2))=Hom(A_1,B_1)\otimes Hom(A_2,B_2)$
\item Composition $\bullet:Hom(B,C)\otimes Hom(A,B)\to Hom(A,C)$, as $\bullet=(\circ_1\otimes \circ_2)\circ(id\otimes c\otimes id)$; where $\circ_i:Hom(B_i,C_i)\otimes Hom(A_i,B_i)\to Hom(A_i,C_i)$ is the composition map of $\mathcal{C}_i$. We are considering $A=(A_1,A_2)$, $B=(B_1,B_2)$ and $C=(C_1,C_2)$.
\end{itemize}
\end{defn}

With this last defintion we have all the tools to define monoidal and symmetric monoidal   $\mathcal{M}$-categories in a natural way. Just rewriting the definition of monoidal and symmetric monoidal  categories using enriched functors and enriched natural transformations instead of regular ones; and tensor of $\mathcal{M}$-categories instead of product of categories.

The last concept on monoidal categories we are going to consider is closed monoidal strutures:

\begin{defn}
Let $(\mathcal{C},\otimes,\mathbb{I},\alpha,\lambda,\mu)$ be a monoidal category, then it is closed in case each functor $-\otimes A:\mathcal{C}\to \mathcal{C}$ has an adjoint functor $[A,-]:\mathcal{C}\to \mathcal{C}$ for every $A\in Obj(\mathcal{C})$; i.e., $\mathcal{C}(B\otimes A,C)\cong \mathcal{C}(B,[A,C])$. The object $[A,C]$ is called the inner morphisms between $A$ and $C$.
\end{defn}

\subsection{Color categories}

In this subsection we define the categories of $G$-graded vector spaces; which is the category of $\mathbb{K}[G]$-comodules; and give them several symmetric structures.

\begin{defn}
A $G$-graded vector space is defined as a vector space $V$ such that $V=\bigoplus_{g\in G}V_g$; where $G$ could be non-abelian.
\end{defn}

\begin{lem}
$\mathbf{Vec}_\mathbb{K}^G=(Vec_\mathbb{K}^G,\otimes,\mathbb{K},\alpha,\lambda,\mu)$ is a monoidal category; where $G$ could be non-abelian.

The objects in this category are $G$-graded vector spaces and for any $V$ and $W$, $G$-graded vector spaces, the arrows are defined as $$Hom(V,W)=\{f:V\to W|\ f\text{ is linear and }\forall g\in G,\ f(V_g)\subseteq W_g\}$$

The other elements in the tuple are defined as:

$$\begin{matrix}\otimes&:&Vec_\mathbb{K}^G\times Vec_\mathbb{K}^G&\to& Vec_\mathbb{K}^G\\
&&(V,W)&\mapsto&V\otimes W\end{matrix}$$

$$\begin{matrix}\rho_V&:&V\otimes\mathbb{K}&\to& V\\
&&v\otimes\mu&\mapsto&\mu v\end{matrix}$$

$$\begin{matrix}\lambda_V&:&\mathbb{K}\otimes V&\to &V\\
&&\mu\otimes v&\mapsto&\mu v\end{matrix}$$

$$\begin{matrix}\alpha_{V,U,W}&:&(V\otimes U)\otimes W&\to &V\otimes(U\otimes W)\\
&&(v\otimes u)\otimes w&\mapsto &v\otimes(u\otimes w)\end{matrix}$$
\end{lem}

\begin{proof}

It is clear that $Vec_\mathbb{K}^G$ is a category, $\alpha$ makes the pentagon diagram \ref{fig:pentagon} commutative, and $\rho$, $\lambda$ and $\alpha$
make the triangle diagram \ref{fig:triangle} commutative; and that the three are isomorphisms for all objects.

What is left to prove is that $-\otimes-$ is a functor.

If we define $|v\otimes w|=|v||w|$, it follows that $V\otimes W=\bigoplus_{g\in G}(V\otimes W)_g$ where $(V\otimes W)_g=\bigoplus_{h\in G}V_h\otimes W_{h^{-1}g}$. In conclusion $\otimes:Vec_\mathbb{K}^G\times Vec_\mathbb{K}^G\to Vec_\mathbb{K}^G$ is well defined for objects.

The tensor on morphisms is defined as $\otimes(f,g)(v\otimes w)=(f\otimes g)(v\otimes w)=(f v)\otimes (g w)$. This is well defined and $\otimes(id_V,id_W)=id_{V\otimes W}$. Also $(f\circ g)\otimes (h\circ t)=(f\otimes h)\circ(g\otimes t)$. Hence $\otimes$ is a functor.
\end{proof}

\begin{rem}The category of $G$-graded vector spaces is just the category of comodules of the group algebra $\mathbb{K}[G]$; the coaction is defined as $v\to v\otimes g$ for $v\in V_g$. \end{rem}

The monoidal structure of the category has been stablished, but the number of symmetric structures that this particular categories admit is big; and so we need to concern ourselves now with some specific symmetries. In particular we will deal with symmetries related to the group structure of $G$.

\begin{defn}
A skew-symmetric bicharacter is a mapping $\chi:G\times G\to \mathbb{K}^\times$ such that:
\begin{itemize}
\item $\chi(ab,c)=\chi(a,c)\chi(b,c)$ and $\chi(a,bc)=\chi(a,b)\chi(a,c)$
\item $\chi(a,b)\chi(b,a)=1$
\end{itemize}

We will define as a color the pair $(G,\chi)$ as above.
\end{defn}

\begin{lem}\label{gg}
Let $(G,\chi)$ be a color, then $\chi(g,g)=\pm 1$.
\end{lem}

\begin{defn}
$G_+=\{g\in G|\ \chi(g,g)=1\}$ are called the even $G$-elements and $G_-=\{g\in G|\ \chi(g,g)=-1\}$ are called the odd $G$-elements. For a $G$-graded vector space $V=(\bigoplus_{g\in G_+} V_g)\oplus(\bigoplus_{g\in G_-} V_g)=V_+\oplus V_-$. 

$V_-$ is called the odd part and $V_+$ is called the even part.
\end{defn}

\begin{lem}
Given a color $C=(G,\chi)$, the natural isomorphism defined on homogeneous elements as $$ \begin{matrix}c^C_{V,W}&:&V_h\otimes W_g&\to&W_g\otimes V_h\\  
&&v\otimes w&\mapsto &\chi(h,g)w\otimes v\end{matrix}$$ gives a symmetric braiding to the monoidal cathegory
$\mathbf{Vec}_\mathbb{K}^G$.
\end{lem}

In this last result, it is mandatory that $G$ be abelian; since otherwise $c^C$ is not a morphism of graded spaces.

Most of the time, by abuse of notation we shall write $\chi(v,w)$ instead of $\chi(|v|,|w|)$ for homogenous elements $v$ and $w$.



\begin{defn}
For a color $(G,\chi)$, denote by $\mathbf{Vec_\mathbb{K}(G,\chi)}$ the symmetric monoidal category given by the monoidal category $\mathbf{Vec}_\mathbb{K}^G$ and the braiding given by $\chi$.

This family of categories are the color categories.
\end{defn}

Examples of the previous categories are:

\begin{itemize}
\item Vector spaces: the color given by $G=\{0\}$ and $\chi(0,0)=1$.
\item $G$-graded vector spaces: the color given by $\chi(g,h)=1$.
\item Super vector spaces: the color given by $G=\mathbb{Z}_2$ and $\chi(a,b)=(-1)^{ab}$.
\end{itemize}

The key property of this categories, that will be used through the paper, is that we can consider this category as a $\mathbb{K}$-linear category:

\begin{lem}
The tensor of $G$-graded vector spaces and all the natural transformations that define the symmetric monoidal structure of $\mathbf{Vec_\mathbb{K}(G,\chi)}$ are multilinear; and hence they define linear mappings on the tensor products. Such structures make $(Vec_\mathbb{K}^G,\widehat{\otimes},\mathbb{K},\hat{\alpha},\hat{\lambda},\hat{\rho},\widehat{c^C})$ into a $\mathbb{K}$-linear category.
\end{lem}

A notation that we shall use extensively is given by the following:

Since the category $\mathbf{Vec_\mathbb{K}(G,\chi)}$ is symmetric, it implies that there is a left action of $S_n$ onto $V^{\otimes n}$ by automorphisms. This action, using the definition of the symmetry $c^C_{A,B}(v\otimes w)=\chi(|v|,|w|)(w\otimes v)$ on homogeneous elements, is given by  $\sigma:V^{\otimes n}\to V^{\otimes n}$ such that $\sigma\cdot x=\sigma(x_1\otimes\dots\otimes x_n)=k(x_{\sigma^{-1} 1}\otimes\dots\otimes x_{\sigma^{-1} n})$ on homogenous elements; where $k\in\mathbb{K}^\times$ depends on $\chi,\ x$ and $\sigma$. 

\begin{defn}
Usually we will denote the non zero element $k$ by $$\chi(x,\sigma)\ \text{ or }\ \chi(x,x_{\sigma^{-1} 1}\otimes\dots\otimes x_{\sigma^{-1} n})$$
\end{defn}

\subsection{Group bicharacters}
\begin{defn}
The linear character of a group $G$ is a group homomorphism $\chi:G\to \mathbb{K}^\times$. This set of group homomorphisms form a group by pointwise multiplication denoted by $G^*$ or $\widehat{G}$, and called the dual group.
\end{defn}

In case the group $G$ is finite and abelian, the bicharacter can be taken from a much bigger selection, as the dual group $G^*$ of linear characters of $G$ is isomorphic to $G$. This is no longer true for non abelian groups. For example $S_5$ has only two non-isomorphic linear characters; and the linear characters and in close relation with the bicharacters.

\begin{lem}
For a finite group $G$, $Bicharacters(G)\cong Hom(Ab(G),G^*)$; where $Ab(G)=G/[G,G]$.
\end{lem}

\begin{proof}
Let $\chi:G\times G\to\mathbb{K}$ be a bicharacter. Then for all $g\in G$, $\chi(g,-):G\to \mathbb{K}^\times$ is a linear character. We have then that $\chi_\rho:G\to G^*$ is a group homomorphism; where $\chi_\rho(g)(h)=\chi(g,h)$.

Consider now a group homomorphism $\rho:G\to G^*$, and $g,\ h\in G$. Then $\rho_\chi:G\times G\to\mathbb{K}^\times$ is a bicharacter; where $\rho_\chi(g,h)=\rho(g)(h)$.

In conclusion $Bicharacters(G)\cong Hom(G,G^*)$. Since $G^*$ is an abelian group, it follows that $Bicharacters(G)\cong Hom(G/[G,G],G^*)$.
\end{proof}

In case $\mathbb{K}=\mathbb{C}$ and the group is topological, one can consider only continuous characters and bicharacters $\chi$ such that $im(\chi)\subseteq \mathbb{S}^1$.

\begin{cor}
In case $G$ is finite, $Bicharacters(G)\cong End(G^*)$; and in case $G$ is abelian $Bicharacter(G)\cong Hom(G,G^*)$.
\end{cor}

Reconsidering our example, $S_5^*\cong \mathbb{Z}_2$ and $End(\mathbb{Z}_2)\cong \mathbb{Z}_2$; concluding that $S_5$ has only two bicharacters.

Another example can be $\mathbb{Z}$ with the discrete topology and field $\mathbb{C}$, where $Bicharacter(\mathbb{Z})\cong \mathbb{S}^1$; since $\chi(n,m)=\chi(1,1)^{nm}$ and $\chi(1,1)\in \mathbb{S}^1$.

If one considers $\mathbb{Z}$ in the more general setting, then $Bicharacter(\mathbb{Z})\cong\mathbb{K}^\times$ following the same argument as before.

If we restrict then to skew-symmetric bicharacters of finite abelian groups, then we have the following results:

\begin{thm}
The only skew symmetric bicharacter of the group $\mathbb{Z}_{2k+1}$ is the trivial one: $\chi(a,b)=1$  for any $a,b\in \mathbb{Z}_{2k+1}$
\end{thm}

\begin{proof}

$$1=\chi(0,0)=\chi(2k+1,2k+1)=\chi(1,1)^{(2k+1)^2}=$$ $$\chi(1,1)^{4k^2+4k+1}=\chi(1,1)(\chi(1,1)^{4(k^2+k)})=\chi(1,1)$$ In the previous computations we have used that $\chi(1,1)\in\{-1,1\}$ as shown in \ref{gg}.

Since $\chi(a,b)=\chi(1,1)^{ab}$ for $a,b\ne 0$ it follows that then $\chi(a,b)=1$ since we know already that $\chi(0,a)=\chi(a,0)=1$.
\end{proof}

\begin{thm}
The only skew symmetric bicharactesr of the group $\mathbb{Z}_{2k}$ is the trivial one and even-odd one: $\chi(a,b)=1$ if $ab$ is even and $\chi(a,b)=-1$ if $ab$ is odd.
\end{thm}

\begin{proof}

$$1=\chi(0,0)=\chi(2k,2k)=\chi(1,1)^{(2k)^2}=\chi(1,1)^{4k^2}$$

Then the element $\chi(1,1)$ can be chosen either as $1$ or as $-1$. If $\chi(1,1)=1$ we get the trivial bicharacter and if $\chi(1,1)=-1$ we get the even-odd bicharacter.
\end{proof}

\subsection{Categorical tools}

The main non-trivial example of color category is that of super vector spaces; or $\mathbf{SVec}_\mathbb{K}$. For this particular category there have been developed some functorial tools that helped on the study of superalgebras, or algebras on $\mathbf{SVec}_\mathbb{K}$. In this subsection we study a generalization of this tools.

Firstly, there is a concept of Grassmann superalgebra as the free associtive, commutative and unital algebra on $\mathbf{SVec}_\mathbb{K}$ generated by an infinite countable dimensional vector space of odd elements.

The generalization of this concept is given by:

\begin{defn}\label{freeCommu}
Consider $V$ a $G$-graded vector space, then $T(V)$ is also $G$-graded. $T(V)$ with the natural product is an associative algebra in the color category of any color $(G,\chi)$.

Consider then the ideal generated by $I(\chi)=\langle u\otimes v-\chi(g,h)v\otimes u|\ v\in V_g,\ u\in V_h \rangle$.

Then $G(V,\chi):=T(V)/I(\chi)$ is commutative in $\mathbf{Vec_\mathbb{K}(G,\chi^{op})}$; and it is $G$-graded since $I(\chi)$ is also graded.

Specify $V= \mathbb{K}<\{x_g^i|\ g\ne 1;\ i\in\mathbb{N}\}\cup\{1\}>$ and define
$Gr^G_\chi=G(V,\chi)$. This is denoted as the Grassmann color algebra of $(G,\chi)$.
\end{defn}

Examples of the previous constructions are:
\begin{itemize}
\item In case of vector spaces $G(V,\chi)=S(V)$; the symmetric algebra.
\item In case of super vector spaces $G(V,\chi)=S(V_0)\otimes G(V_1)$; where $G(V_1)$ is the Grassmann super algebra generated by $V_1$.
\end{itemize}

Notice that in the case of super vector spaces $\chi(a,b)=\chi(b,a)$ and so $\chi^{op}=\chi$. This is the reason why the Grassmann superalgebra is commutative in $\mathbf{SVec}_\mathbb{K}$.

The next step is to define a generalization of the Grassmann envelope. To reach this concept we first need this three other functors:

\begin{defn}
$(-)_0:\mathbf{Vec(G,\chi)}\to \mathbf{Vec}$ defined as $(V)_0$ is the vector space of homogenous elements in $V$ with degree $e\in G$; the neutral element.

Given an arrow $f\in Hom(V,W)$ in $\mathbf{Vec(G,\chi)}$; $f_0\in Hom(V_0,W_0)$ is defined as $f_0=f|_{V_0}$.
\end{defn}

\begin{defn}
$Forget:\mathbf{Vec(G,\chi)}\to\mathbf{Vec(G,1)}$ defined as the forgetful functor that forgets the symmetry; but no the grading.
\end{defn}

\begin{defn}
$Gr^G_\chi: \mathbf{Vec(G,1)}\to \mathbf{Vec(G,\chi)}$ defined as $Gr^G_\chi(V)=Gr^G_\chi\otimes V$.

Given an arrow $f\in Hom(V,W)$ in $\mathbf{Vec(G,1)}$; $$Gr^G_\chi(f)\in Hom(Gr^G_\chi(V),Gr^G_\chi(V))$$ is defined as $$Gr^G_\chi(f)=id(Gr^G_\chi)\otimes f$$
\end{defn}

\begin{prop}
The functors $(-)_0$, $Forget$ and $Gr_\chi^G$ are monoidal functors.
\end{prop}

\begin{proof}
It is widely known that $(-)_0:\mathbf{Vec(G,\chi)}\to\mathbf{Vec}$ is an actual functor and is an easy exercise to check that $i:V_0\otimes W_0\to (V\otimes W)_0$ is a natural transformation between the functors $(-)_0\otimes (-)_0$ and $(-\otimes -)_0$; where $i$ is defined as the inclusion since $(V\otimes W)_0=\bigoplus_{g\in G}V_g\otimes W_{g^{-1}}$. Hence we only need to prove that it is monoidal with the natural transformation $i$ and morphism $\phi=id(\mathbb{K}):\mathbb{K}\to\mathbb{K}$.

This means that the diagrams \ref{fig:monoidal functor1}, \ref{fig:monoidal functor2} and \ref{fig:monoidal functor3} must commute. This is also a trivial exercise.

Trivially, $Forget:\mathbf{Vec(G,\chi)}\to\mathbf{Vec(G,1)}$ is a functor. It is monoidal since $Forget(V\otimes W)=Forget(V)\otimes Forget(W)$ and $Forget(\mathbb{K})=\mathbb{K}$.

To prove that $Gr^G_\chi: \mathbf{Vec(G,1)}\to \mathbf{Vec(G,\chi)}$ is monoidal we need to find an appropiate $i:Gr^G_\chi(V)\otimes Gr^G_\chi(W)\to Gr^G_\chi(V\otimes W)$ and $\phi:\mathbb{K}\to Gr^G_\chi(\mathbb{K})$ that make the diagrams \ref{fig:monoidal functor1}, \ref{fig:monoidal functor2} and \ref{fig:monoidal functor3} commute.

This mappings are given by $i(g\otimes v,h\otimes w)=(gh)\otimes (v\otimes w)$ and $\phi(a)=1\otimes a$.
\end{proof}

\begin{defn}\label{Colenv}
Define then $Env_\chi^{G}=(-)_0\circ Gr^G_\chi\circ Forget:\mathbf{Vec(G,\chi)}_\mathbb{K}\to\mathbf{Vec}_\mathbb{K}$.
\end{defn}

\begin{prop}
The functor $Env^G_\chi$ is symmetric monoidal with associated natural transformations $i:Env^G_\chi(V)\otimes Env^G_\chi(W)\to Env^G_\chi(V\otimes W)$ given by $i(g\otimes v,h\otimes w)=gh\otimes (v\otimes w)$ and $\phi:\mathbb{K}\to Env^G_\chi(\mathbb{K})$ given by $\phi(a)=1\otimes a$.
\end{prop}

\begin{proof}
By definition, $Env^G_\chi$ is a monoidal functor since it is composition of monoidal functors. It rests to show that it is in fact symmetric.

Let $t_1=(g\otimes v)\otimes (h\otimes w)\in Env^G_\chi(V)\otimes Env^G_\chi(W)$; then $t_2=i(t_1)=gh\otimes (v\otimes w)\in Env^G_\chi(V\otimes W)$. Hence $Env^G_\chi(c)(t_2)=gh\otimes (\chi(|v|,|w|)(w\otimes v))=gh\otimes(\chi(g^{-1},h^{-1})(w\otimes v))=\chi(g,h)gh\otimes (w\otimes v)$.

On the other hand, $t_3=c(t_1)=(h\otimes w)\otimes (g\otimes v)$ and $i(t_3)=hg\otimes (w\otimes v)=\chi(g,h)gh\otimes(w\otimes v)$.

In conclusion the diagram \ref{fig:braided monoidal functor} commutes.
\end{proof}

The functor  $Env^\chi_{G}$ reduces to the grassmann envelope in the case of $\mathbf{SVec}$ and is symmetric monoidal, hence algebras in $\mathbf{Vec(G,\chi)}$ are transported to algebras in $\mathbf{Vec}$.

The trick of this functor is translating the grading and the symmetry from the categorical structure to the algebraic structure, so the categorial structure transforms to $\mathbf{Vec}$ without forgeting the group $G$ and its bicharacter $\chi$ that can be recovered from the algebraic structure.






\section{Algebraic Definitions}

In this section, we define what it is meant as an algebra in some category  $\mathbf{Vec_\mathbb{K}(G,\chi)}$ and some definitions related to them.

\subsection{Introduction}

An algebraic structure in some category is defined as expected:

\begin{defn}
In a category $\mathcal{C}$ with a tensor product $\otimes$, an algebra is defined as an object $A$ and a family of morphisms $\{f_i:A^{\otimes n_i}\to A\}$.
\end{defn}

The most interesting categories are those which are symmetric monoidal which let us define identities with the use of the action of the symmetric groups. In fact, our previously defined monoidal functor $Env^G_\chi$ lets us give a definition of some algebra belonging to a variety:

\begin{defn}
Given a class of algebras $\mathcal{D}$ in $\mathbf{Vec}$, we will define the class of algebras $\mathcal{D}_{G,\chi}$ as the algebras $(B,\{f_i:B^{\otimes_\chi n_i}\to B\})$  in the category $\mathbf{Vec_\mathbb{K}(G,\chi)}$ such that $(Env^G_\chi(B),\{Env^G_\chi(f_i)\circ i^{n_i}: Env^G_\chi(B)^{\otimes n_i}\to Env^G_\chi(B)\}$ is an algebra in the class $\mathcal{D}$.

In this case $i^{n_i}: Env^G_\chi(B)^{\otimes n_i}\to Env^G_\chi(B^{\otimes_\chi n_i})$ is the natural transformation resulting of applying $i:Env^G_\chi(A)\otimes Env^G_\chi(B) \to Env^G_\chi(A\otimes_\chi B)$ several times; for $A=B^{\otimes q}$ where  $q\in\mathbb{N}$.
\end{defn}

\begin{lem}\label{Recon}
Let $\mathcal{D,C}$ be two classes of algebras in the category of vector spaces and $\mathcal{D}_{G,\chi},\mathcal{C}_{G,\chi}$ the classes defined previously in the category $\mathbf{Vec_\mathbb{K}(G,\chi)}$. Let $R_{G,\chi}:\mathcal{C}_{G,\chi}\to\mathcal{D}_{G,\chi}$ be a linear functor defined for every color. If there is a linear natural transformation $N_{G,\chi}:R_{\{0\},1}\circ  Env_{G}^\chi\to Env_{G}^\chi\circ R_{G,\chi}$ defined for every color, then any polynomial equation fulfilled by $R_{\{0\},1}(Env_\chi^G(M))$ may give at least one equation that is fulfilled by $R_{G,\chi}(M)$ for every algebra $M$ in $\mathcal{C}_{G,\chi}$; in case $\forall y\in R_{G,\chi}(M),\ \exists g\in Gr^G_\chi$ such that $g\otimes y\in im(N_{G,\chi})$.
\end{lem}

\begin{proof}
Let $\theta(x_1,\ldots,x_n)=0$ be some polynomial equation fullfiled by $R(Env^\chi_{G}(M))$. Then $N_M:R\circ Env^\chi_{G}(M)\to Env^\chi_{G}\circ R_{G,\chi}(M)$ is a morphism of algebras given by the natural transformation.

It follows then that $N_M\circ \theta(x_1,\ldots,x_n)=0$. Since $N_M$ is morphism and $\theta$ is polynomial, there exists $\bar{\theta}$ such that $\bar{\theta}(N_M(x_1),\ldots,N_M(x_n))=0$.

We can choose elements $x_i$ such that $N_M(x_i)=g_i\otimes y_i$ for any $y_i\in R_{G,\chi}(M)$, by the last hypothesis.

Expanding the equation, $$\bar{\theta}(N_M(x_1),\ldots,N_M(x_n))=$$ $$\bar{\theta}(g_1\otimes y_1,\dots, g_n\otimes y_n)=\sum_{i}h_i\otimes p_i=0$$ where $p_i$ are ponomials in $R_{G,\chi}(M)$ valued in $(y_1,\ldots,y_n)$ and the set $\{h_i\}$ is linearly independent. Then $p_i=0$ for every $i$ by linear independency. Hence the polynomials $p_i=0$ are polynomial equations for $R_{G,\chi}(M)$.
\end{proof}

\begin{rem}
Iif the polynomial equations are $n$-linear; i.e. $$\theta(x_1,\ldots,\lambda x_i,\ldots,x_n)=\lambda\theta(x_1,\ldots,x_i,\ldots,x_n),$$  the result of the process gives one polynomial equation. In almost all examples in this paper, this is the case.
\end{rem}

An example of this construction:

\begin{prop}
Consider $\operatorname{U}$ the enveloping algebra of a color Lie algebra; and $U$ the enveloping algebra of a Lie algebra, then there is a natural transformation $n:U\circ Env^G_\chi\to Env^G_\chi\circ \operatorname{U}$ such that $\forall y\in \operatorname{U}(L),\ \exists g\in Gr_\chi$ with $g\otimes y\in n_L(U(Env^G_\chi(L)))$ for any $L$ Lie color algebra.
\end{prop}

\begin{proof}
Let $L$ be any Lie color algebra, define $n_L:U(Env^G_\chi(L))\to Env^G_\chi(U(L))$ as $n_L(g\otimes a)=g\otimes a$ and extend as an algebra morphism: $n_L((g_1\otimes a_1)\dots (g_n\otimes a_n))=(g_1\dots g_n)\otimes (a_1\dots a_n)$.

With this definition, $n:U\circ Env^G_\chi\to Env^G_\chi\circ U$ is a natural transformation and $\forall y\in \operatorname{U}(L),\ \exists g\in Gr_\chi$ with $g\otimes y\in n(U(Env^G_\chi(L)))$ for any $L$ Lie color algebra; as we wanted.
\end{proof}

Since this functor, $Env^G_\chi$, is not comonoidal, we cannot recover directly the coproduct for $U(L)$; where $L$ is a Lie color algebra; as we can recover the product from $n$ and $U(Env^G_\chi(L))$.

The coproduct of $U(L)$; for a color algebra $L$, must be a map $\Delta_\chi:U(L)\to U(L)\otimes U(L)$ such that the diagram \ref{fig:coproduct} commutes. Since $Env^G_\chi$ is an injective linear functor, if such a map exists, it is unique.
For it to exists, it must happend that $i\circ(n_L\otimes n_L)\circ\Delta((g_1\otimes a_1)\dots (g_n\otimes a_n))=(g_1\dots g_n)\otimes A$, where then $\Delta_\chi(a_1\dots a_n)=A\in U(L)\otimes U(L)$.

It can be easily shown that everything holds and $\Delta_\chi$ exists; since $i\circ(n_L\otimes n_L)\circ\Delta$ is an algebra morphism and it holds for primitives; which generate the whole algebra.

\begin{figure}
	\centering
\begin{tikzpicture}
\node (P0) at (180-35:3cm) {$\operatorname{U}(Env^G_\chi(L))$};
\node (P1) at (35:3cm) {$\operatorname{U}(Env^G_\chi(L))\otimes \operatorname{U}(Env^G_\chi(L))$} ;
\node (P2) at (35+180:3cm) {$Env^G_\chi(\operatorname{U}(L))$};
\node (P4) at (0:3cm) {$Env^G_\chi(\operatorname{U}(L))\otimes Env^G_\chi(\operatorname{U}(L))$};
\node (P3) at (-35:3cm) {$Env^G_\chi(\operatorname{U}(L)\otimes \operatorname{U}(L))$};
\draw
(P0) edge[->,>=angle 90] node[above] {$\Delta$} (P1)
(P0) edge[->,>=angle 90] node[right] {$n_L$} (P2)
(P1) edge[->,>=angle 90] node[right] {$n_L\otimes n_L$} (P4)
(P4) edge[->,>=angle 90] node[right] {$i$} (P3)
(P2) edge[->,>=angle 90] node[below] {$Env^G_\chi(\Delta_\chi)$} (P3);
\end{tikzpicture}
\caption{Coproduct definition diagramm}
\label{fig:coproduct}
\end{figure}

\subsection{Malcev Algebras}

In this subsection we deal with the definition of Malcev algebra in a color category. To get such an algebra, we only need to apply our definition to the variety of Malcev algebras; and the result is:

\begin{defn}
For a color $(G,\chi)$, define color-alternative algebras as the class of color-algebras $A=\bigoplus_{g\in G}A_g$ such that:
 $\forall a,b,c\in  A$ homogenous elements; $(a,b,c)=-\chi(a,b)(b,a,c)$ and $(a,b,c)=-\chi(b,c)(a,c,b)$ where $(a,b,c)$ is the associator.

Notice that it may happen that $(a,a,b)\ne 0$ in case $\chi(a,a)=-1$.
\end{defn}

Let $(A,a\cdot b)$  be a color-alternative algebra and consider the new product $[a,b]_\chi=ab-\chi(a,b)ba$

then $$[a,b]_\chi=-\chi(a,b)[b,a]_\chi\text{\  \ and}$$
$$[[x,z]_\chi,[y,w]_\chi]_\chi\chi(y,z)=$$ $$[[[x,y]_\chi,z]_\chi,w]_\chi+\chi(x,y)\chi(x,z)\chi(x,w)[[[y,z]_\chi,w]_\chi,x]_\chi+$$ $$\chi(y,z)\chi(x,z)\chi(y,w)\chi(x,w)[[[z,w]_\chi,x]_\chi,y]_\chi+\chi(z,w)\chi(y,w)\chi(x,w)[[[w,x]_\chi,y]_\chi,z]_\chi$$ 

which reduce to the linearization of the Malcev identity or Malcev super-identity chosing the proper color $(G,\chi)=(\mathbb{Z}_2,\chi_s)$, hence the name linearized Malcev color-identity. 

It is important to work with linearized identities because it may happend that $\chi(x,x)=-1$ and hence some of the identities fail to be equivalent in case they permute two copies of the same element. An example is $[J(x,y,z),y]$; which permutes $y$ with itself.

\begin{lem}
The algebras that fulfill the linearized Malcev color-identity and the antisymmetry color-identity, are the Malcev color-algebras.
\end{lem}
\begin{proof}
Direct application of the definition.
\end{proof}

From any algebra we can obtain a color-alternative algebra:

\begin{defn}
Let $A=\bigoplus_{g\in G} A_g$ be a color-algebra and define $N_{color}(A)=\bigoplus_{g\in G}\{x\in A_g|\ (x,y,z)=-\chi(x,y)(y,x,z)\text{ and }(y,z,x)=-\chi(z,x)(y,x,z),$ $\forall y,z\in\bigcup_{h\in G} A_h\}$. This subspace is called the color-alternative nucleous of $A$.
\end{defn}

\begin{lem}
$N_{color}(A)^-$ is a Malcev color-algebra.
\end{lem}

\begin{proof}
It is known that $N_{alt}(Env^G_\chi(A))^-$ is a Malcev algebra.

What we need to prove is that the linearized Malcev color identity is fulfilled by this construction of that $Env^G_\chi(N_{color}(A)^-)$ is a Malcev algebra.

Define the following map: $f:N_{alt}(Env^G_\chi(A))\to Env^G_\chi(A)$ as the inclusion; and consider $g\otimes a\in im(f)$.

Then $g\otimes a\in N_{alt}(Env^G_\chi(A))$ and $\forall h\otimes b,\ q\otimes c$; $(g\otimes a,h\otimes b,q\otimes c)=-(h\otimes b,g\otimes a,q\otimes c)$. In conclusion, $$ghq\otimes (a,b,c)=-hgq\otimes (b,a,c)$$ $$ghq\otimes (a,b,c)=ghq\otimes (-\chi(a,b)(b,a,c))$$

It follows then that $im(f)= Env^G_\chi(N_{color}(A))$.

In conclusion, $f$ is an isomorphism from $N_{alt}(Env^G_\chi(A))$ to $Env^G_\chi(N_{color}(A))$; hence $ Env^G_\chi(N_{color}(A))$ is a Malcev algebra.
\end{proof}

\subsection{Universal Enveloping Algebra}

In this section $\mathbb{K}\{V\}$ indicates the free algebra generated by $V$. 

From a color-Malcev algebra $M$, construct the quotient algebra $U(M)=\mathbb{K}\{M\}/I$ for $I$ the ideal generated by $\langle x\otimes y-\chi(x,y)y\otimes x-[x,y]_\chi,(x,a,b)+\chi(x,a)(a,x,b),(a,x,b)+\chi(x,b)(a,b,x)|\ x,y\in M;\ a,b\in \mathbb{K}\{V\}$ homogeneous elements$\rangle$. Let $i:M\to U(M)$ be $i(x)=x+I$. We shall prove next that $U(M)$ is the functor adjoint to $N_{color}(-)^-$ applied to $M$.

\begin{thm}
Let $(M,[x,y])$ be a Malcev color-algebra. Let also $A$ be an algebra and $f:M\to N_{color}(A)^-$ a Malcev color-algebra morphism. Then exists a unique $U(f):U(M)\to A$ algebra morphism such that $U(f)\circ i = f$.
\end{thm}

\begin{proof}
Define $F:\mathbb{K}\{M\}\to A$ as an algebra morphism and $F|_M=f$; which can be done in a unique way. In conclusion we can define $\widetilde{F}:U(M)\to A$ as $\widetilde{F}(x+I)=F(x)$ for $x\in \mathbb{K}\{M\}$, because $I\subseteq ker(F)$ by the definition of $im(f)\subseteq N_{color}(A)^-$ as a Malcev color-algebra and $f$ as a Malcev color-algebra morphism.
$\widetilde{F}\circ i=f$ by definition and concluding: $U(f)=\widetilde{F}$.

Uniqueness follows from the fact that $\overline{M}$ generates $U(M)$ as an algebra, and hence $U(f)$ is completely defined by its image on $M$.
\end{proof}

\begin{cor}
$Hom_{alg}(U(M),A)\cong Hom_{malcev}(M,N_{color}(A)^-)$
\end{cor}

This construction as the adjoint functor through a quotion of algebras doesn't let us check wether $i(M)\cap I=\{0\}$ or not. The last part of this paper solves this problem.
In the case of Malcev algebras, this problem has been solved in \cite{MalcevUni} and $U(M)$ has a bialgebra structure called Hopf-Moufang algebra.

This class of bialgebras have been defined over the category of vector spaces. First we will define $H$-bialgebra and afterwards Hopf-Moufang algebras as a subclass of this algebras.

\begin{thm}\label{HomAlg}
In a braided closed monoidal $\mathbb{K}$-linear category $\mathcal{C}$, let $(H,\Delta,\epsilon,\mu)$ be a coassociative, counital bialgebra; then $\mathcal{C}(H,[H,H])$ is an associative unital algebra given the product $a*b=\bullet\circ(a\otimes b)\circ\Delta$; where $\bullet:[H,H]\otimes [H,H]\to [H,H]$ is the composition map.
\end{thm}

\begin{proof}
Let $a,b,c:H\to [H,H]$, then we need to show that $(a*b)*c=a*(b*c)$ and to find an unit.

$$(a*b)*c=(\bullet\circ(a\otimes b)\circ\Delta)*c=$$
$$\bullet\circ((\bullet\circ(a\otimes b)\circ\Delta)\otimes c)\circ \Delta=$$
$$\bullet\circ(\bullet\otimes id)\circ((a\otimes b)\otimes c)\circ(\Delta\otimes id)\circ \Delta$$

$$a*(b*c)=a*(\bullet\circ(b\otimes c)\circ\Delta)=$$
$$\bullet\circ(a\otimes(\bullet\circ(b\otimes c)\circ\Delta))\circ \Delta=$$
$$\bullet\circ(id\otimes\bullet)\circ(a\otimes(b\otimes c))\circ(id\otimes\Delta)\circ \Delta$$

The result follows from the associativity of composition, tensor products and the coassociativity of the coproduct.

$T=\lambda\circ(\epsilon\otimes id):H\otimes H\to H$ and by the adjoint functors $-\otimes H$ and $[H,-]$ there is a unique map $u:H\to [H, H]$ such that $T=ev\circ(u\otimes id)$. By the definition of counit $T\circ\Delta=id:H\to H$.

Similarly to the associativity computations, one realizes that $u*a=a*u=a$ by use of the definition of the composition $\bullet:[H,H]\otimes [H,H]\to[H,H]$ and the uniqueness of maps by the universal properties of adjoint functors.
\end{proof}

\begin{defn}
In a symmetric closed monoidal $\mathbb{K}$-linear category, an algebra $(H,\mu)$ has two maps, called adjoint maps $L:H\to [H,H]$ and $R:H\to [H,H]$ that make the diagram \ref{fig:adj} commutes.

\begin{figure}
	\centering
\begin{tikzpicture}
\node (P0) at (180:1.5cm) {$H\otimes H$};
\node (P1) at (90:2cm) {$ [H,H] \otimes H$} ;
\node (P3) at (0:1.5cm) {$H$};
\node (P4) at (270:2cm) {$H\otimes[H,H]$} ;
\draw
(P0) edge[->,>=angle 90] node[left] {$L\otimes id$} (P1)
(P1) edge[->,>=angle 90] node[right] {$ev$} (P3)
(P0) edge[->,>=angle 90] node[left] {$ id\otimes R$} (P4)
(P4) edge[->,>=angle 90] node[right] {$ev\circ c$} (P3)
(P0) edge[->,>=angle 90] node[above] {$\mu$} (P3);
\end{tikzpicture}
\caption{Adjoint maps}
\label{fig:adj}
\end{figure}

$L$ is denoted left adjoint and $R$, right adjoint. Their existence is explained by the universal property of the adjoint functors. Since the category is symmetric, $R$ also exists by the use of the isomorphism between $H\otimes [H,H]$ and $[H,H]\otimes H$; as the diagram \ref{fig:adj} shows.
\end{defn}

\begin{defn}
In a braided closed monoidal $\mathbb{K}$-linear category $\mathcal{C}$
a counital coassociative bialgebra $(H,\Delta,\cdot,\epsilon)$ is called an $H$-bialgebra if its adjoint maps $L$ and $R$ have bilateral inverses in the algebra $\mathcal{C}(H,[H,H])$ defined in \ref{HomAlg}.
\end{defn} 

The simplest example of an $H$-bialgebra is a Hopf algebra, in fact Hopf algebras can be defined as associative unital $H$-bialgebras where there exists $S:H\to H$ such that $L^{-1}=L\circ S$ and $R^{-1}=R\circ S$.

\begin{defn}
In a  symmetric closed monoidal $\mathbb{K}$-linear category $\mathcal{C}$, a Hopf-Moufang algebra is a unital $H$- bialgebra $(H,\mu,\Delta,\epsilon,u)$ in $\mathcal{C}$ such that the diagram \ref{fig:Moufang} commutes and there is $S:H\to H$ such that $L^{-1}=L\circ S$ and $R^{-1}=R\circ S$.

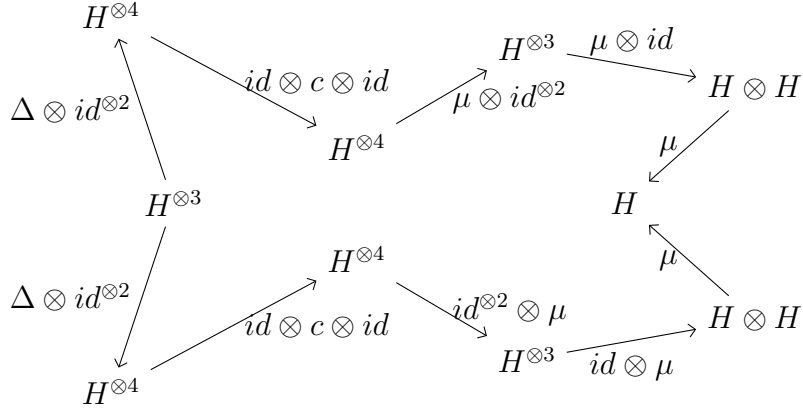
\begin{figure}
	\centering
\begin{tikzpicture}
\node (P0) at (180:3.5cm) {$H^{\otimes 3}$};
\node (P5) at (0:2.5cm) {$H$};
\node (P4) at (150:5cm) {$H^{\otimes 4}$};
\node (P6) at (145:1.3cm) {$H^{\otimes 4}$};
\node (P41) at (-150:5cm) {$H^{\otimes 4}$};
\node (P61) at (-145:1.3cm) {$H^{\otimes 4}$};
\node (P7) at (300:2.4cm) {$H^{\otimes 3}$};
\node (P8) at (340:4.5cm) {$H\otimes H$};
\node (P9) at (60:2.4cm) {$H^{\otimes 3}$};
\node (P10) at (20:4.5cm) {$H\otimes H$};
\draw
(P4) edge[->,>=angle 90] node[right] {$id\otimes c\otimes id$} (P6)
(P41) edge[->,>=angle 90] node[right] {$id\otimes c\otimes id$} (P61)
(P0) edge[->,>=angle 90] node[left] {$ \Delta\otimes id^{\otimes 2}$} (P4)
(P0) edge[->,>=angle 90] node[left] {$ \Delta\otimes id^{\otimes 2}$} (P41)
(P61) edge[->,>=angle 90] node[right] {$id^{\otimes 2}\otimes\mu$} (P7)
(P7) edge[->,>=angle 90] node[below] {$id\otimes\mu$} (P8)
(P8) edge[->,>=angle 90] node[left] {$\mu$} (P5)
(P6) edge[->,>=angle 90] node[right] {$\mu\otimes{ id^{\otimes 2}}$} (P9)
(P9) edge[->,>=angle 90] node[above] {$\mu\otimes id$} (P10)
(P10) edge[->,>=angle 90] node[left] {$\mu$} (P5);
\end{tikzpicture}
\caption{Moufang relations' diagram of a bialgebra}
\label{fig:Moufang}
\end{figure}

In the case $\mathcal{C}=\mathbf{Vec(G,\chi)}$, the diagram \ref{fig:Moufang} translates into the equation $$\sum\chi(|u_{(2)}|,|v|)((u_{(1)}v)u_{(2)})w=\sum\chi(|u_{(2)}|,|v|)u_{(1)}(v(u_{(2)}w))$$

\end{defn}

\subsection{Lie and Hopf algebras with Triality}

The relation between connected Hopf algebras and Lie algebras as equivalent categories is well known. This subsection particularizes this relation to those algebras with triality.

\begin{defn}
In a symmetric monoidal $\mathbb{K}$-linear category $\mathcal{C}$, a Lie algebra with triality is a Lie algebra $(L,\mu)$ in $\mathcal{C}$ endowed with two automorphisms $\sigma_2:L\to L$ and $\sigma_3:L\to L$ such that $\langle\sigma_2,\sigma_3\rangle \cong \mathcal{S}_3$ and that $(id+\sigma_3+\sigma_3^2)\circ(\sigma_2-id)=0$.
\end{defn}

\begin{defn}
In a symmetric monoidal $\mathbb{K}$-linear category $\mathcal{C}$, a Hopf algebra with triality is a Hopf algebra $(H,\mu,\Delta,\epsilon,u)$ in $\mathcal{C}$ endowed with two automorphisms $\sigma_2:H\to H$ and $\sigma_3:H\to H$ such that $\langle\sigma_2,\sigma_3\rangle \cong \mathcal{S}_3$ and that $$\sum P(x_{(1)})\sigma_3(P(x_{(2)}))\sigma_3^2(P(x_{(3)}))=\epsilon(x)1$$ where $P(x)=\sum \sigma_2(x_{(1)})S(x_{(2)})$.
\end{defn}

\subsubsection{From Malcev algebra to Lie algebra with triality}

The first step to construct the enveloping algebra of a Malcev color algebra will be to construct a Lie color algebra with triality from the Malcev algebra.

\begin{lem}
Let $A$ be some color-algebra. Consider then $L_a:A\to A$ as $L_a(x)=ax$ and $R_b:A\to A$ as $R_b(x)=\chi(b,x)xb$.

Leat $a,b\in N_{color}(A)$, then they fulfill the following equations:

\begin{enumerate}
\item $L_{ab}=L_aL_b+[R_a,L_b]_\chi$ and $L_{ba}=L_bL_a+[L_b,R_a]_\chi$
\item $[R_a,L_b]_\chi=[L_a,R_b]_\chi$
\item $[L_a,L_b]_\chi=L_{[a,b]_\chi}-2[R_a,L_b]_\chi$ and $[R_a,R_b]_\chi=-R_{[a,b]_\chi}-2[R_a,L_b]_\chi$
\end{enumerate}
\end{lem}
\begin{proof}
The proof is just a colored version of the results on \cite{MalcevUni} and \cite{Ebar}.
\end{proof}

In light of the previous result we can define a new Lie color-algebra:

\begin{defn}\label{LieTri}
Let $M$ be a Malcev color-algebra and consider the following color vector space $V=\oplus_{g\in G}\langle \mathcal{L}_a,\mathcal{R}_a|\ a\in M_g\rangle$. Consider the $L(M)$ as the Lie color-algebra generated by this vector space and with the following restrictions for $a,b\in M$ and $k\in\mathbb{K}$:
\begin{itemize}
\item $  \mathcal{L}_{a+b}=\mathcal{L}_a+\mathcal{L}_b$ and $\mathcal{L}_{ka}=k\mathcal{L}_a$. 
\item $  \mathcal{R}_{a+b}=\mathcal{R}_a+\mathcal{R}_b$ and $\mathcal{R}_{ka}=k\mathcal{R}_a$.
\item $[\mathcal{R}_a,\mathcal{L}_b]_\chi=[\mathcal{L}_a,\mathcal{R}_b]_\chi$
\item $[\mathcal{L}_a,\mathcal{L}_b]_\chi=\mathcal{L}_{[a,b]_\chi}-2[\mathcal{R}_a,\mathcal{L}_b]_\chi$
\item  $[\mathcal{R}_a,\mathcal{R}_b]_\chi=-\mathcal{R}_{[a,b]_\chi}-2[\mathcal{R}_a,\mathcal{L}_b]_\chi$
\end{itemize}
\end{defn}

The main feature of this construction is that this Lie algebra has a triality, as we shall prove, and contains the original Malcev color-algebra.

In the article \cite{Ebar}, where the universal enveloping algebra of a Malcev superalgebra is built, the constructions and results follow those of the Malcev algebra case, done in \cite{MalcevUni}.

It is clear that this constructions are functorial and that they do not change when considering categories more complex than that of vector spaces.

\begin{thm}
Let $M$ be a Malcev color-algebra, then $L(M)$ fulfills the following equations:

\begin{itemize}
\item $[T_a,T_b]_\chi=\dfrac{1}{3}ad_{[a,b]_\chi}+\dfrac{2}{3}D_{a,b}$
\item $[ad_a,T_b]_\chi=T_{[a,b]_\chi}$
\item $[ad_a,ad_b]_\chi=-ad_{[a,b]_\chi}+2D_{a,b}$
\item $[D_{a,b},ad_{c}]_\chi=ad_{D_{a,b}(c)}$
\item $[D_{a,b},T_{c}]_\chi=-T_{D_{a,b}(c)}$
\item $[D_{a,b},D_{c,d}]_\chi=D_{D_{a,b}(c),d}+\chi(a,c)\chi(b,c)D_{c,D_{a,b}(d)}$
\end{itemize}

and hence has a $\mathbb{Z}_2$ graduation $L(M)= L(M)_-\oplus L(M)_+$ given by $L(M)_-=\langle T_a|a\in M\rangle$ and $L(M)_+=\langle ad_a,D_{a,b}|a,b\in M\rangle$; where $T_a=\mathcal{L}_a+\mathcal{R}_a$ $ad_a=\mathcal{L}_a-\mathcal{R}_a$ and $D_{a,b}=ad_{[a,b]}-3[\mathcal{L}_a,\mathcal{R}_b]_\chi$.
\end{thm}

\begin{proof}
Define $F:L(Env_\chi^G(M))\to Env_\chi^G(L(M))$ as $$F(T_{g\otimes a})=g\otimes T_a$$ $$F(ad_{g\otimes a})=g\otimes ad_a$$ $$F(D_{g\otimes a, h\otimes b})=gh\otimes D_{a,b}$$ and extend by linearity and multiplicativity.

The map $F$ is a well defined Lie morphism:

$$[F(\mathcal{L}_{g\otimes a}),F(\mathcal{R}_{h\otimes b})]_\chi=[g\otimes \mathcal{L}_a,h\otimes \mathcal{R}_{b}]_\chi=gh\otimes [\mathcal{L}_a,\mathcal{R}_{b}]_\chi=$$  $$=gh\otimes[\mathcal{R}_a,\mathcal{L}_{b}]_\chi=F([\mathcal{R}_{g\otimes a},\mathcal{L}_{h\otimes b}]_\chi)$$

$$[F(\mathcal{L}_{g\otimes a}),F(\mathcal{L}_{h\otimes b})]_\chi=[g\otimes \mathcal{L}_a,h\otimes \mathcal{L}_{b}]_\chi=gh\otimes (\mathcal{L}_{[a,b]_\chi}-2[\mathcal{R}_a,\mathcal{L}_b]_\chi)=$$ $$F(\mathcal{L}_{gh\otimes [a,b]_\chi}-2[\mathcal{R}_{g\otimes a},\mathcal{L}_{h\otimes b}]_\chi)=F([\mathcal{L}_{g\otimes a},\mathcal{L}_{h\otimes b}]_\chi)$$

$$[F(\mathcal{R}_{g\otimes a}),F(\mathcal{R}_{h\otimes b})]_\chi=[g\otimes \mathcal{R}_a,h\otimes \mathcal{R}_{b}]_\chi=gh\otimes (-\mathcal{R}_{[a,b]_\chi}-2[\mathcal{R}_a,\mathcal{L}_b]_\chi)=$$ $$F(-\mathcal{R}_{gh\otimes [a,b]_\chi}-2[\mathcal{R}_{g\otimes a},\mathcal{L}_{h\otimes b}]_\chi)=F([\mathcal{R}_{g\otimes a},\mathcal{R}_{h\otimes b}]_\chi)$$

Applying the result \ref{Recon} with $R=L$ and $N=F$ we get all the equations in the theorem since they appear for $L(M)$ in the category of vector spaces in \cite{MalcevUni}. 
\end{proof}

\begin{thm}
Let $M$ be a Malcev color-algebra and construct $L(M)$ as explained in \ref{LieTri}. Then there are automorphisms $\sigma_2,\sigma_3:L(M)\to L(M)$ defined on a set of generators as $$\sigma_2(\mathcal{L}_a)=-\mathcal{R}_a$$ $$\sigma_2(\mathcal{R}_a)=-\mathcal{L}_a$$ $$\sigma_3(\mathcal{L}_a)=\mathcal{R}_a$$ $$\sigma_3(\mathcal{R}_a)=-\mathcal{L}_a-\mathcal{R}_a$$ that make $L(M)$ a Lie color-algebra with triality.
\end{thm}

\begin{proof}
Trivially $\sigma_2,\sigma_3:L(M)\to L(M)$ exist and define automorphisms, since they satisfy the equations on the definition of $L(M)$.

Also, the group generated by these automorphisms is $\mathcal{S}_3$. So it sufficies to check if $(id+\sigma_3+\sigma_3^2)(id-\sigma_2)=0$. A simple calculation shows that $(id+\sigma_3+\sigma_3^2)(x)=0$ for $x=\mathcal{L}_a$ and $x=\mathcal{R}_a$.

Since $(id-\sigma_2)(\mathcal{L}_a)=\mathcal{L}_a+\mathcal{R}_a=(id-\sigma_2)(\mathcal{R}_a)$ it follows that $(id+\sigma_3+\sigma_3^2)(id-\sigma_2)=0$ for $\mathcal{L}_a$ and $\mathcal{R}_a$. Hence $(id+\sigma_3+\sigma_3^2)(id-\sigma_2)=0$ for $T_a$ and $ad_a$. It rests to check if it vanishes on $D_{a,b}$ or equvalently on $[\mathcal{L}_a,\mathcal{R}_b]_\chi$.

$(id-\sigma_2)([\mathcal{L}_a,\mathcal{R}_b]_\chi)=[\mathcal{L}_a,\mathcal{R}_b]_\chi-[\mathcal{R}_a,\mathcal{L}_b]_\chi=0$.
\end{proof}

\subsubsection{Enveloping algebra of Lie algebra with triality}

The next step is to show how the enveloping algebra of a Lie algebra with triality is a Hopf algebra with triality. The key is our proof resides in the grassmann envelope and the known result in the case of vector spaces.

\begin{prop}
The enveloping algebra of a Lie color algebra with triality is a Hopf color algebra with triality.
\end{prop}

\begin{proof}
Let $L$ be a Lie color algebra with triality $\sigma_2,\ \sigma_3$. Then $Env_\chi^G(L)$ is a Lie algebra and $\rho_2=Env_\chi^G(\sigma_2),\ \rho_3=Env_\chi^G(\sigma_3)$ are a triality for this Lie algebra:

$$(id+\rho_3+\rho_3^2)(id-\rho_2)=Env_\chi^G((id+\sigma_3+\sigma_3^2)(id-\sigma_2))=0$$

It follows that $\operatorname{U}(Env_\chi^G(L))$ is a Hopf algebra with triality as it is shown in \cite{MalcevTriality}[Theorem 4.4]. The natural transformation $n:\operatorname{U}\circ Env_\chi^G\to Env_\chi^G\circ\operatorname{U}$ stablishes that $\operatorname{U}(L)$ fulfills the identities of a Hopf color algebra with triality.
\end{proof}

\subsubsection{From Hopf algebra with triality to Hopf-Moufang algebra}

We remeber that for a Malcev color algebra $M$, the Lie color algebra $L(M)$ has a $\mathbb{Z}_2$ graduation and $L(M)_-$ has a very special property:

\begin{prop}
The map $\phi:L(M)_-\to M$ where $T_a\mapsto a$ is a linear isomorphism.
\end{prop}

This last step consists in giving a functorial construction of a Hopf-Moufang color algebra from a Hopf color algebra with triality and to show how this functor applied to $\operatorname{U}(L(M))$ gives as a result the universal enveloping algebra of $M$; the Malcev color algebra.

\begin{defn}
Given a Hopf color algebra $H$ with triality $\sigma_2,\ \sigma_3$, we define $\mathcal{MH}(H)=\{P(x)|\ x\in H\}$; where $P(x)=\sum \sigma_2(x_{(1)})S(x_{(2)})$.\end{defn}

\begin{prop} $\mathcal{MH}(H)$ with product $u * v=\sum\chi(|u_{(2)}|,|v|) \sigma_3^2(S(u_{(1)}))v\sigma_3(S(u_{(2)}))$ on homogenous elements; together with the unit, counit and coproduct inhereted by $H$ is a bialgebra.
\end{prop}

\begin{proof}
Consider the associative algebra $Env_\chi^G(H)$. Firstly, $\rho_2=Env_\chi^G(\sigma_2)$ and $\rho_3=Env_\chi^G(\sigma_3)$ are automorphisms by the properties of functors.

On the algebra $Env_\chi^G(H)$ the coproduct is applied as $Env_\chi^G(\Delta)(g\otimes x)=g\otimes \Delta(x)$; also the counit is $Env_\chi^G(\epsilon)(g\otimes x)=g\otimes \epsilon(x)$.

Define $PEnv_\chi^G(H)$ as the subalgebra of $Env_\chi^G(H)$ generated by $Env_\chi^G(Prim H)$ and the unit element. First, we see that $Env_\chi^G(\Delta)(PEnv_\chi^G(H))\subseteq PEnv_\chi^G(H)\otimes PEnv_\chi^G(H)$; and that $i(PEnv_\chi^G(H)\otimes PEnv_\chi^G(H))\subseteq PEnv_\chi^G(H\otimes H)$. Even more, $\rho_2$ and $\rho_3$ can be considered automorphisms of the algebra $PEnv_\chi^G(H)$.

\begin{figure}
	\centering
\begin{tikzpicture}
\node (P2) at (35+180:3cm) {$PEnv^G_\chi(H)$};
\node (P1) at (35:3cm) {$PEnv^G_\chi(H)\otimes PEnv^G_\chi(H)$};
\node (P3) at (-35:3cm) {$PEnv^G_\chi(H\otimes H)$};
\draw
(P2) edge[->,>=angle 90] node[below] {$\widehat{\Delta}$} (P1)
(P1) edge[->,>=angle 90] node[right] {$i$} (P3)
(P2) edge[->,>=angle 90] node[below] {$Env^G_\chi(\Delta_\chi)$} (P3);
\end{tikzpicture}
\caption{Coproduct for the Grassmann envelope of $H$}
\label{fig:coproduct2}
\end{figure}

To give $PEnv^G_\chi(H)$ a bialgebra structure, we need a coproduct and counit. The natural structure would come considering the diagramm \ref{fig:coproduct2}.

The problem in this setting is the fact that there is not a unique $\widehat{\Delta}$ that makes the diagramm \ref{fig:coproduct2} commutative; there are several such maps. They arise by the fact that this algebra is not a UFD and could be $g\otimes xy=(g_1\otimes x)(g_2\otimes y)=(h_1\otimes x)(h_2\otimes y)$; where $h_1\ne g_1$ and $h_2\ne g_2$. Nevertheless, we consider one of those coproducts.

As counit, we consider $\epsilon(g\otimes x)=g\otimes \epsilon(x)$. Note that if $x\in\mathbb{K}$, then $|g|=0$; and in case $|g|\ne 0$, then $\epsilon(x)=0$.

By construction, $PEnv^G_\chi(H)$ is then a bialgebra; with two special automorphisms. And the following equations hold:

$$\sum (g\otimes x)_{(1)}Env_\chi^G(S)((g\otimes x)_{(2)})=\sum g_{(1)}g_{(2)}\otimes x_{(1)}S(x_{(2)})=$$ $$g\otimes \sum\chi(x,x_{(1)}x_{(2)})x_{(1)}S(x_{(2)})=g\otimes \epsilon(x)$$

Hence $PEnv_\chi^G(H)$ is a Hopf algebra, with two special automorphisms.

$$\sum P((g\otimes x)_{(1)})\rho_3P((g\otimes x)_{(2)})\rho_3^2P((g\otimes x)_{(3)})=$$ $$\sum g_{(1)}g_{(2)}g_{(3)}\otimes P( x_{(1)})\rho_3P( x_{(2)})\rho_3^2P( x_{(3)})=$$ $$g\otimes \sum\chi(x,x_{(1)}x_{(2)}x_{(3)}) P( x_{(1)})\rho_3P( x_{(2)})\rho_3^2P( x_{(3)})=g\otimes \epsilon(x)$$

In conclusion, $PEnv^G_\chi(H)$ is a Hopf algebra with triality.

In this algebra $\mathcal{MH}(PEnv^G_\chi(H))$ can be defined following \cite{MalcevTriality}[Theorem 3.3] and it is a Hopf-Moufang algebra with product:

$$(g\otimes u)*(h\otimes v)=\sum\sigma_3^2(S((g\otimes u)_{(1)}))(h\otimes v)\sigma_3(S((g\otimes u)_{(2)}))$$

But then, the natural transformation $n:\mathcal{MH}(PEnv_\chi^G(H))\to Env_\chi^G(\mathcal{MH}(H))$ results in:

$$((g_1\dots g_n)\otimes (u_1\dots u_n))*((h_1\dots h_m)\otimes (v_1\dots v_m)=(gh)\otimes (u* v)$$ $$=\sum(g_{(1)}hg_{(2)})\otimes\sigma_3^2(S(u_{(1)}))v\sigma_3(S(u_{(2)}))$$ $$=\sum(g_{(1)}g_{(2)}h)\otimes\chi(|u_{(2)}|,|v|)\sigma_3^2(S(u_{(1)}))v\sigma_3(S(u_{(2)}))$$

$$(gh)\otimes(u*v-\sum\chi(|u_{(2)}|,|v|)\chi(u,u_{(1)}u_{(2)})\sigma_3^2(S(u_{(1)}))v\sigma_3(S(u_{(2)})))=0$$

The conclusion that follows, rewriting the notation of the cocommutative coproduct for the color category, is that defining $u*v$ as in the hypothesis $\mathcal{MH}(H)$ becomes a Moufang-Hopf color algebra; with the unit, counit and coproduct inherited from $H$.

\end{proof}

\begin{cor}
$\mathcal{MH}(\operatorname{U}(L(M)))$ is a Moufang-Hopf algebra.
\end{cor}

This algebra is our candidate for the universal enveloping algebra of the Malcev color algebra $M$. The following results all lead to proof that fact.

First, we find a basis of $\mathcal{MH}(\operatorname{U}(L(M)))$ as a vector space.

\begin{lem}\label{basis}
Let $L$ be a Lie color algebra and $\lambda:S_3\to Aut(L)$ an action of $S_3$ as automorphisms of $L$; with $\lambda((12))=\sigma$.

Then $\mathcal{MH}(\operatorname{U}(L))=\{P(x)|\ x\in \operatorname{U}(L)\}=span\langle a_n\bullet(a_{n-1}\bullet\dots\bullet(a_2\bullet a_1))|\ a_i\in E(-1;\sigma),\ n\in\mathbb{N}\rangle$; where $P(x)=\sum \sigma(x_{(1)})S(x_{(2)})$ and $a\bullet b=ab+\chi(|a|,|b|)ba$ on homogeneous elements.
\end{lem}

\begin{proof}
We will procede by induction on the degree filtration on $U(L)$. The base case is $a\in L$.

$$P(a)=\sigma(a)-a\in E(-1;\sigma)$$

Assume know that it is true up to some natural number $k\in\mathbb{N}$; and consider $a_1\dots a_{k+1}\in U(L)$.
We may assume that $a_i\in E(-1;\sigma)\cup E(1;\sigma)$; since $\sigma^2=Id$.

By induction, we only need to consider the cases where $a_i\in E(-1,\sigma)$ for $i\leq k$. Hence we are reduced to two cases: $\sigma(a_{k+1})=a_{k+1}$ or $\sigma(a_{k+1})=-a_{k+1}$.

In the first case $$P(xa_{k+1})=$$ $$\sum\chi(|x_{(2)}|,|\displaystyle{a_{k+1}}|)\sigma(x_{(1)})\sigma(a_{k+1})S(x_{(2)})-$$ $$\sum\chi(|x_{(2)}|,|a_{k+1}|)\sigma(x_{(1)})a_{k+1}S(x_{(2)})=0$$ and we are done.

In the latter case $$P(a_1x)=$$ $$\sum\sigma(a_1)\sigma(x_{(1)})S(x_{(2)})-\sum\chi(|a_1|,|x|)\sigma(x_{(1)})S(x_{(2)})a_1=-a_1\bullet P(x)$$ and we are done.
\end{proof}

The previous bases can be constructed using the product $*$ defined on $\mathcal{MH}(\operatorname{U}(L(M)))$; instead of that of $\operatorname{U}(L(M))$; as the following lemma shows.

\begin{lem}\label{eqProd}
In the previous settings for the Lie color algebra $L(M)$, $T_a*u+\chi(|a|,|u|)u*T_a=T_au+\chi(|a|,|u|)uT_a$ for any $u\in \mathcal{MH}(\operatorname{U}(L(M)))$
\end{lem}

\begin{proof}
By  \cite{MalcevTriality}[Lemma 5.2], $(g\otimes T_a)*(h\otimes u)+(h\otimes u)*(g\otimes T_a)=(g\otimes T_a)(h\otimes u)+(h\otimes u)(g\otimes T_a)$; in $Env_\chi^G(\operatorname{U}(L(M)))$ and $Env_\chi^G(\mathcal{MH}(\operatorname{U}(L(M))))$

$$(g\otimes T_a)*(h\otimes u)+(h\otimes u)*(g\otimes T_a)=(gh)\otimes (T_a* u)+(hg)\otimes (u* T_a)=$$ $$(gh)\otimes (T_a* u+\chi(|a|,|u|)u* T_a)$$

$$(g\otimes T_a)(h\otimes u)+(h\otimes u)(g\otimes T_a)=(gh)\otimes (T_a u)+(hg)\otimes (u*T_a)=$$ $$(gh)\otimes (T_au+\chi(|a|,|u|)u T_a)$$

In conclusion, it follows that $T_a*u+\chi(|a|,|u|)u*T_a=T_au+\chi(|a|,|u|)uT_a$, applying the natural transformation $n:\operatorname{U}\circ Env_\chi^G\to Env_\chi^G\circ\operatorname{U}$.
\end{proof}

\section{Main result}

All the previous results can be added to prove the equivalence between Malcev color algebras and connected color Hopf-Moufang algebras.

\begin{thm}
$\mathcal{MH}(\operatorname{U}(L(M)))\cong \operatorname{U}(M)$ and $Prim(U(M))\cong M$.
\end{thm}
\begin{proof}
$$P(T_a)=\sum \sigma_2(T_a)_{(1)} S(T_a)_{(2)}=\sigma_2(T_a)- T_a=-2T_a$$

Hence $T_a\in \mathcal{MH}(\operatorname{U}(L(M)))$
and is a primitive element of $\mathcal{MH}(\operatorname{U}(L(M)))$; so they are part of the alternative nucleous.

In $Env_\chi^G(\mathcal{MH}(\operatorname{U}(L(M))))$, from \cite{MalcevTriality}[Theorem 5.3]
$(g\otimes T_a) *( h\otimes T_b)-(h\otimes T_b) *(g\otimes T_a)=-[g\otimes T_a,h\otimes T_b]$; which transforms itself to $(gh)\otimes (T_a*T_b-\chi(|a|,|b|)T_b*T_a)=(gh)\otimes (-[T_a,T_b]_\chi)$ when applied the natural transformation $n: \mathcal{MH}\circ\operatorname{U}\circ Env_\chi^G\to Env_\chi^G\circ\mathcal{MH}\circ\operatorname{U}$.

In conclusion $T_a*T_b-\chi(|a|,|b|)T_b*T_a=-[T_a,T_b]_\chi$ and hence the map $a\mapsto -T_a$, by the universal property of $\operatorname{U}$, extends to an algebra morphism $\phi:\operatorname{U}(M)\to \mathcal{MH}(\operatorname{U}(L(M)))$

By the lemma \ref{basis}, $\mathcal{MH}(\operatorname{U}(L(M)))$ is spanned by $\{T_{a_n}\bullet(\dots(T_{a_2}\bullet T_{a_1}))|\ a_i\in M; n\in\mathbb{N}\}$ from wich follows that $\phi$ is surjective using also lemma \ref{eqProd}. 

Also, $M\cap ker(\phi\circ i)=\{0\}$ and hence we have the PBW theorem for $\operatorname{U}(M)$; where $i:M\to \operatorname{U}(M)$ is the map $i(m)=\overline{m}$; it follows then that $Prim(U(M))\cong M$. Given $\{a_i\}_i$, an ordered basis of homogenous elements for $M$, the basis for $\operatorname{U}(M)$ can be thought as $\{a_{i_n}\bullet(\dots\bullet(a_{i_2}\bullet a_{i_1}))|\ i_j\leq i_{j+1}$ if $a_{i_j}$ is an even element and $i_j<i_{j+1}$ if $a_{i_j}$ is an odd element$\}$.

When $\phi$ is applied to the basis of $\operatorname{U}(M)$, the image of the elements result in the set $\{(-1)^nT_{a_{i_n}}\bullet (\dots\bullet(T_{a_{i_2}}\bullet T_{a_{i_1}}))\}$, using again the lemma \ref{eqProd}; which is linearly independent in $\operatorname{U}(L(M))$. In conclusion, $\phi$ is injective.

There it follows that $\mathcal{MH}(\operatorname{U}(L(M)))\cong \operatorname{U}(M)$ through $\phi$.
\end{proof}

\bibliographystyle{plainurl}
\bibliography{ArticlesBib}

\end{document}